\documentclass[10pt,a4paper]{amsart}

\usepackage{amstext,amssymb,amsopn,amsthm,mathrsfs,amsfonts,amsmath,amsxtra}

\usepackage{dsfont,multirow,comment,enumerate,bm,float}

\usepackage{graphicx,color}

\allowdisplaybreaks

\newtheorem{theorem}{Theorem}[section]

\newtheorem{proposition}[theorem]{Proposition}

\newtheorem{lemma}[theorem]{Lemma}
\newtheorem{remark}[theorem]{Remark}

\def\z{j}

\title[New identities involving Bessel zeros]
	{On new identities involving zeros of Bessel functions}

\author[B.\ Langowski]{Bartosz Langowski}

\address{Bartosz Langowski, \newline
Department of Mathematics and Physical Sciences, \newline
Franciscan University of Steubenville, \newline
1235 University Blvd.,\ Steubenville, OH 43952, USA
}
\email{blangowski@franciscan.edu}

\author[A.\ Nowak]{Adam Nowak}

\address{Adam Nowak, \newline
			Institute of Mathematics, \newline
      Polish Academy of Sciences, \newline
      \'Sniadeckich 8, 00--656 Warszawa, Poland 
}
\email{adam.nowak@impan.pl}

\begin{document}


\begin{abstract}
We derive new identities involving zeros of the Bessel function $J_{\nu}$ and some related functions.
These are special cases of more general identities obtained in this note, which might also be of interest.
\end{abstract}

\maketitle
\thispagestyle{empty}

\section*{Addendum after publication}
After this paper had been published in J.\ Math.\ Anal.\ Appl., the author of [K. Urbanowicz, \emph{Infinite series based on Bessel zeros}, Appl.\ Sci.\ 13 (2023), 12932, 28 pp.] informed us about his paper and an overlap with his results. More precisely, Formulas (F9) and (F10) from our paper appear as Formulas (A11) and (A14) in the above-mentioned paper. Our results are independent, and the proofs are different. In fact, our paper had been finished and submitted to a journal in November 2023, before the paper by Urbanowicz was published.

Moreover, thanks to Urbanowicz' paper, we became aware of the reference [G.N. Afanasiev, \emph{Closed expressions for some useful integrals involving Legendre functions and sum rules for zeroes of Bessel functions}, J.\ Comput.\ Phys.\ 85 (1989), 245--252], where essentially Formulas (F1)--(F4) from our paper appear on page 251 there, however (F2) with a misprint, and (F4) with an error that was corrected by Urbanowicz in the paper quoted above, see (A5) there. The methods used by Afanasiev are different from our approach.

We refer to Urbanowicz' paper for many more series involving zeros of Bessel functions, interesting motivations, and a long list of references relevant to the topic. We are grateful to the author for bringing his paper to our attention.

\footnotetext{
\emph{2020 Mathematics Subject Classification:} primary 33C10; secondary 33C45, 42C05\\
\emph{Key words and phrases:} 
zeros of Bessel function, Fourier-Bessel system, Fourier-Dini system.
}

\section{Introduction} \label{sec:intro}

Let $\nu > -1$ and denote by $\z_n^{\nu}$, $n \ge 1$, the sequence of successive strictly positive zeros of
the Bessel function $J_{\nu}$, cf.\ \cite[Chap.\,XV]{watson}.
There is no explicit expression for $\z_n^{\nu}$ except for the two special cases when
$\nu = \pm 1/2$ in which one has $\z_n^{-1/2} = \pi(n-1/2)$ and $\z_n^{1/2}=\pi n$.
Nevertheless, there are many striking summation identities involving the zeros of $J_{\nu}$ where apart from the zeros
everything else is explicit. A classic example here is Raighley's formula (cf.\ \cite[Chap.\,XV, Sec.\,15{$\cdot$}51]{watson})
\begin{equation} \label{Ray}
\sum_{k=1}^{\infty} \frac{1}{(\z_k^{\nu})^2} = \frac{1}{4(\nu+1)}.
\end{equation}
Actually, the above series is computable explicitly if one replaces $2$ by any other positive even power, see
\cite{meiman,kishore,sneddon,CDPV}.
Another classic identity involving the zeros of $J_{\nu}$ is the Mittag-Leffler expansion
(cf.\ \cite[Chap.\,XV, Sec.\,15{$\cdot$}41]{watson})
$$
\frac{J_{\nu+1}(x)}{J_{\nu}(x)} = \sum_{k=1}^{\infty} \frac{2x}{(\z_{k}^{\nu})^2-x^2},
$$
where, in particular, the left-hand side vanishes with the choice $x = \z_{n}^{\nu+1}$, $n \ge 1$.

Further examples of this kind are formulas found by Calogero \cite{Cal,Cal2}
\begin{align}
\sum_{k \ge 1, k \neq n} \frac{(\z_n^{\nu})^2}{(\z_k^{\nu})^2 - (\z_n^{\nu})^2} & = \frac{\nu+1}2, \qquad n \ge 1, \label{cal} \\
\sum_{k \ge 1, k \neq n} \frac{(\z_n^{\nu})^4}{[(\z_k^{\nu})^2 - (\z_n^{\nu})^2]^2} & =
	\frac{(\z_n^{\nu})^2 - (\nu+1)(\nu+5)}{12}, \qquad n \ge 1. \nonumber
\end{align}
Curiously enough, the research of Calogero had some physical background, actually
it was mainly a spin-off of the study of certain integrable many-body problems, cf.\ \cite{Caletal} and references there.
It is also possible to compute analogous series with the numerator and the denominator of the summand in \eqref{cal}
raised separately to various integer positive powers, see \cite{Cal2,Cal4,Caletal}.
Other formulas in this spirit can be found in \cite{sneddon}, \cite[Sec.\,5.7.33]{prud2}, \cite{baricz} (see also \cite{Ang})
and \cite{CDPV}, among others.

It is perhaps interesting that all the above displayed formulas were indispensable in our recent study devoted to some questions
in harmonic analysis of Fourier-Bessel expansions \cite{LN}. This inspired the present note in which we
derive several up to our best knowledge new identities including, for instance,
$$
\sum_{k \ge 1} \frac{(\z_k^{\nu})^2}{[ (\z_k^{\nu})^2 - (\z_n^{\nu+1})^2 ]^2}
= \sum_{k \ge 1} \frac{(\z_k^{\nu+1})^2}{[(\z_k^{\nu+1})^2 - (\z_n^{\nu})^2]^2}
= \frac{1}4,
\qquad n \ge 1.
$$
Comparing to most of other known formulas,
our identities involve zeros of more than one Bessel or related function (in the two cases above,
$J_{\nu}$ and $J_{\nu+1}$).
The method we use is rather elementary, but relies on a quite deep fact that the Fourier-Bessel and Fourier-Dini
systems form orthogonal bases in $L^2((0,1),dx)$.

\section{Preliminaries} \label{sec:pre}

In this section we first briefly describe the general notation used in the paper.
Then we present for further reference various facts and formulas related to Bessel functions.
Finally, we give basic information on Fourier-Bessel and Fourier-Dini systems.

\subsection{Notation} 

Throughout this note we use fairly standard notation.
By $\langle \cdot, \cdot \rangle$ we denote the usual inner product in $L^2((0,1),dx)$.
For a $\textsf{condition}$ that can have logical value true or false we shall denote
$$
\chi_{\{\textsf{condition}\}} :=
	\begin{cases}
	1, & \textrm{if} \;\, \textsf{condition} \;\, \textrm{is true}, \\
	0, & \textrm{if} \;\, \textsf{condition} \;\, \textrm{is false}.
	\end{cases}
$$

When denoting zeros of Bessel and related functions, we choose to write the index indicating the order as superscript,
i.e.\ $\z_n^{\nu}$, $\z_n^{\nu,H}$ rather than $\z_{n,\nu}$ or $\z_{n,\nu,H}$.
This is to make the notation more compact and stick to notation of other quantities or objects depending on the order.

\subsection{Facts and formulas concerning the Bessel functions $J_{\nu}$ and $I_{\nu}$} \label{sec:Bes}
For the material presented in this section we refer to Watson's monograph \cite{watson}; for an easy
online access to many facts and formulas related to Bessel functions see \cite[Chap.\,10]{handbook}.

Let $J_{\nu}$ be the Bessel function of the first kind and order $\nu$ and let
$I_{\nu}$ be the modified Bessel function of the first kind and order $\nu$. In this paper, for our purpose,
we shall essentially always consider $J_{\nu}$ and $I_{\nu}$ as functions on the positive half-line. Further, throughout the paper
we assume that the parameter $\nu > -1$.

For further reference, we list some useful identities
\begin{align}
J_{\nu-1}(x) + J_{\nu+1}(x) & = \frac{2\nu}x J_{\nu}(x), & I_{\nu-1}(x) - I_{\nu+1}(x) & = \frac{2\nu}x I_{\nu}(x),
\label{id1}\\
x J'_{\nu}(x) - \nu J_{\nu}(x) & = -x J_{\nu+1}(x), & x I'_{\nu}(x) - \nu I_{\nu}(x) & = x I_{\nu+1}(x),
\label{id3}\\
\big[ x^{-\nu} J_{\nu}(x) \big]' & = -x^{-\nu} J_{\nu+1}(x), &  \big[ x^{-\nu} I_{\nu}(x)\big]' & = x^{-\nu} I_{\nu+1}(x). \label{id5}
\end{align}
We also note the asymptotics
\begin{equation} \label{as1}
J_{\nu}(x) = \mathcal{O}(x^{\nu}), \qquad I_{\nu}(x) = \mathcal{O}(x^{\nu}), \qquad x \to 0^+,
\end{equation}
and the special cases
\begin{equation*}
J_{-1/2}(x) = \sqrt{\frac{2}{\pi x}} \cos x, \qquad J_{1/2}(x) = \sqrt{\frac{2}{\pi x}} \sin x.
\end{equation*}
In general, $J_{\nu}$ expresses directly via elementary functions if and only if $\nu$ is an odd half-integer.

Both $J_{\nu}$ and $I_{\nu}$ are smooth on $(0,\infty)$ and $I_{\nu}$ is strictly positive there,
but $J_{\nu}$ has infinitely many isolated zeros in $(0,\infty)$. We denote the sequence of successive strictly positive zeros
of $J_{\nu}$ by $\{\z_n^{\nu} : n \ge 1\}$.
The derivative $J'_{\nu}$ also has infinitely many isolated zeros in $(0,\infty)$. We denote them$^{\dag}$, in an ascending order, by
$\{(\z_n^{\nu})' : n \ge 1\}$.
\footnote{$\dag$ We do not make an exception appearing sometimes in the literature, see e.g.\ \cite[{\S}10.21(i)]{handbook},
that when $\nu=0$ the zeros of $J'_0$ are counted from $x=0$.}

For a parameter $H \in \mathbb{R}$ we consider the auxiliary functions
$$
J_{\nu,H}(x) := x J'_{\nu}(x) + H J_{\nu}(x) \qquad \textrm{and} \qquad
I_{\nu,H}(x) := x I'_{\nu}(x) + H I_{\nu}(x).
$$
By convention, we define $J_{\nu,\infty} := J_{\nu}$ and $I_{\nu,\infty} := I_{\nu}$.
Observe that making use of \eqref{id3} and \eqref{id1} we get
\begin{align} \label{id6}
J_{\nu,H}(x) & = x^{1-H} \big[ x^{H} J_{\nu}(x)\big]' = (H+\nu) J_{\nu}(x) - x J_{\nu+1}(x) = (H-\nu) J_{\nu}(x) + x J_{\nu-1}(x),\\
I_{\nu,H}(x) & = x^{1-H} \big[ x^{H} I_{\nu}(x)\big]' = (H+\nu) I_{\nu}(x) + x I_{\nu+1}(x) = (H-\nu) I_{\nu}(x) + x I_{\nu-1}(x).
\label{id66}
\end{align}
The function $J_{\nu,H}$ has infinitely many isolated zeros in $(0,\infty)$. We denote their sequence in an ascending order
of magnitude by $\{\z_n^{\nu,H} : n \ge 1\}$.
The function $I_{\nu,H}$ has no zeros in $(0,\infty)$ when $H + \nu \ge 0$, and exactly one
strictly positive zero otherwise$^{\ddag}$ which we denote by $\z_0^{\nu,H}$
(note that$^{\natural}$ when $\nu+H < 0$, $\pm i \z_0^{\nu,H}$ are the only imaginary zeros of $(\cdot)^{-\nu} J_{\nu,H}$,
cf.\ \cite[Chap.\,XVIII, Sec.\,18{$\cdot$3}]{watson}).
In case $H+\nu = 0$, we also set $\z_0^{\nu,H} := 0$.
By convention, we define $\z_n^{\nu,\infty} := \z_n^{\nu}$.
\footnote{$\ddag$ This can immediately be seen from \eqref{id66} and the Mittag-Leffler expansion
$\frac{x I_{\nu+1}(x)}{I_{\nu}(x)} = \sum_{k=1}^{\infty} \frac{2}{1+(\z_{k}^{\nu}/ x)^2}$.}
\footnote{$\natural$ It is easy to check that $(ix)^{-\nu}J_{\nu,H}(ix) = x^{-\nu}I_{\nu,H}(x)$ for $x>0$
and $H \in \mathbb{R} \cup \{\infty\}$.}
Observe that, in view of \eqref{id6}, for the special values of $H=-\nu, \nu, 0$ one has
(in the second identity we consider $\nu > 0$)
\begin{equation} \label{zeros}
\z_n^{\nu,-\nu} = \z_n^{\nu+1}, \qquad \z_n^{\nu,\nu} = \z_n^{\nu-1}, \qquad \z_n^{\nu,0} = (\z_n^{\nu})',
\qquad n \ge 1, \qquad \textrm{and} \quad \z_0^{\nu,0} = (\z_0^{\nu})' \quad \textrm{for} \;\; \nu < 0,
\end{equation}
where $(\z_0^{\nu})'$ is the only strictly positive zero of $I'_{\nu}$ in case $\nu < 0$.

We will also need the relations, see \eqref{id1} and \eqref{id3},
\begin{equation} \label{id20}
J_{\nu+1}\big(\z_n^{\nu-1}\big) = \frac{2\nu}{\z_n^{\nu-1}} J_{\nu}\big(\z_n^{\nu-1} \big), \qquad n \ge 1, \quad \nu > 0,
\end{equation}
\begin{equation} \label{id30}
J_{\nu+1}\big((\z_n^{\nu})'\big) = \frac{\nu}{(\z_n^{\nu})'} J_{\nu}\big((\z_n^{\nu})' \big), \qquad n \ge 1,
\qquad 
\qquad
I_{\nu+1}\big((\z_0^{\nu})'\big) = - \frac{\nu}{(\z_0^{\nu})'} I_{\nu}\big((\z_0^{\nu})' \big) \quad \textrm{if} \;\; \nu < 0.
\end{equation}

There is a vast literature devoted to zeros of Bessel and related functions,
which covers also functions like $J_{\nu,H}$. The basic reference is Watson's monograph \cite{watson}.
For more recent developments concerning zeros of $J_{\nu,H}$ and motivations for their study
see for instance \cite{landau} and references given there.
In particular, various series involving zeros of $J_{\nu,H}$ were computed, see e.g.\ \cite{sneddon,Cal3,Cal5}.

\subsection{Fourier-Bessel and Fourier-Dini systems} \label{ssec:FBFD}

Define the following 
quantities:
$$
c_n^{\nu} := \frac{\sqrt{2}}{|J_{\nu+1}(\z_n^{\nu})|}, \qquad
c_n^{\nu,H} := \frac{\sqrt{2}}{|J_{\nu}(\z_n^{\nu,H})|} \frac{\z_n^{\nu,H}}{\sqrt{\big(\z_n^{\nu,H}\big)^2-\nu^2+H^2}}, \qquad n \ge 1,
$$
and, in case $\nu + H < 0$, define also
$$
c_0^{\nu,H} := \frac{\sqrt{2}}{I_{\nu}(\z_0^{\nu,H})} \frac{\z_0^{\nu,H}}{\sqrt{\big(\z_0^{\nu,H}\big)^2 + \nu^2-H^2}}.
$$
By convention, we set $c_n^{\nu,\infty} := c_n^{\nu}$, $n \ge 1$.
Note that the constants $c_n^{\nu,H}$, $H \in \mathbb{R}\cup \{\infty\}$, $n \ge \chi_{\{H+\nu \ge 0\}}$,
are well-defined and strictly positive. This follows from evaluation of $(c_n^{\nu,H})^{-2}$
as certain strictly positive integrals, see the comments preceding Theorem \ref{thm:onb}.

Next, define the following functions on the interval $(0,1)$:
$$
\psi_n^{\nu}(x) := c_n^{\nu} \sqrt{x} J_{\nu}\big(\z_n^{\nu}x\big), \qquad
\psi_n^{\nu,H}(x) := c_n^{\nu,H} \sqrt{x} J_{\nu}\big(\z_n^{\nu,H}x\big), \qquad n \ge 1.
$$
In addition, in case $\nu + H \le 0$, define
$$
\psi_0^{\nu,H}(x) :=
	\begin{cases}
		c_0^{\nu,H} \sqrt{x} I_{\nu}\big(\z_0^{\nu,H}x\big), & \;\; \textrm{if} \;\; H+\nu < 0, \\
		\sqrt{2(\nu+1)} x^{\nu+1/2}, & \;\; \textrm{if} \;\; H+\nu = 0.
	\end{cases}
$$
By convention, we set $\psi_n^{\nu,\infty} := \psi_n^{\nu}$, $n \ge 1$.

The result stated below is well-known, see \cite[Chap.\,XVIII]{watson} and \cite{H}.
The difficult part of the proof is completeness of the two systems, cf.\ \cite{H}.
On the other hand, orthogonality and values of the normalizing constants $c_n^{\nu,H}$ are deduced from Lommel's integrals, cf.\
\cite[Chap.\,V, Sec.\,5$\cdot$11]{watson}. Convenient references to the formulas needed here are the following:
\cite[Sec.\,1.8.3, Form.\,10]{prud2}, \cite[Sec.\,1.11.5, Form.\,2]{prud2}, 
\cite[Sec.\,1.8.1, Form.\,21]{prud2} 
(orthogonality) and
\cite[Sec.\,1.8.3, Form.\,11]{prud2}, \cite[Sec.\,1.11.3, Form.\,4]{prud2} (values of the normalizing constants).
\begin{theorem} \label{thm:onb}
Let $\nu > -1$ and $H \in \mathbb{R}$ be fixed. Then the Fourier-Bessel and the Fourier-Dini systems
$$
\big\{ \psi_n^{\nu} : n \ge 1 \big\} \qquad \qquad \textrm{and} \qquad \qquad \big\{ \psi_n^{\nu,H} : n \ge \chi_{\{H+\nu > 0\}} \big\},
$$
respectively, are orthonormal bases in $L^2((0,1),dx)$.
\end{theorem}
Observe that we can incorporate the Fourier-Bessel system into the family of the Fourier-Dini systems by allowing $H=\infty$.
This has a deeper background, since in some aspects the Fourier-Bessel system can be seen as the limiting case of the Fourier-Dini
system as $H \to \infty$, cf.\ \cite[Chap.\,XVIII, Sec.\,18{$\cdot$}11]{watson}.

The next result is also known and can be easily verified by a direct computation with the aid of \eqref{id5}.
It says that the Fourier-Bessel and the Fourier-Dini systems consist of eigenfunctions of the Bessel differential operator
\begin{equation} \label{id8}
\mathbb{L}^{\nu}f(x) := -\frac{d^2}{dx^2}f(x) - \frac{1/4-\nu^2}{x^2}f(x)
	= - x^{-\nu-1/2} \frac{d}{dx} \Big( x^{2\nu+1} \frac{d}{dx} \big[ x^{-\nu-1/2}f(x)\big] \Big).
\end{equation}
\begin{proposition} \label{prop:eigen}
Let $\nu > -1$ and $H \in \mathbb{R} \cup \{\infty\}$. Then
$$
\mathbb{L}^{\nu} \psi_n^{\nu,H} = (-1)^{\chi_{\{n=0\}}}\big(\z_n^{\nu,H}\big)^2 \psi_n^{\nu,H}, \qquad n \ge \chi_{\{H+\nu > 0\}}.
$$
\end{proposition}
Theorem \ref{thm:onb} and Proposition \ref{prop:eigen} are crucial in what follows.
Note that in the context of Theorem \ref{thm:onb} Parseval's identity can be written as
$$
\|f\|_{L^2((0,1),dx)}^2 = \sum_{k \ge \chi_{\{H+\nu > 0\}}} \big|\big\langle f,\psi_k^{\nu,H}\big\rangle\big|^2,
	\qquad f \in L^2((0,1),dx),
$$
where $\nu > -1$ and $H \in \mathbb{R} \cup \{\infty\}$.

\section{General result} \label{sec:gen}

For $H_1, H_2 \in \mathbb{R}\cup \{\infty\}$ and $n \ge \chi_{\{H_1+\nu  \ge 0\}}$, $k \ge \chi_{\{H_2+\nu  \ge 0\}}$,
define
$$
b_{n,k}^{\nu,H_1,H_2} := c_{n}^{\nu,H_1} c_{k}^{\nu,H_2} \Big[ \z_{n}^{\nu,H_1} J_{\nu+1}\big(\z_{n}^{\nu,H_1}\big)
	J_{\nu}\big(\z_k^{\nu,H_2}\big)
		- \z_k^{\nu,H_2} J_{\nu}\big(\z_n^{\nu,H_1}\big) J_{\nu+1}\big(\z_k^{\nu,H_2}\big)\Big],			
$$
where we use the following convention:
in case $n=0$ the Bessel functions $J_{\nu}$ and $J_{\nu+1}$ which are evaluated at $\z_0^{\nu,H_1}$ must be replaced
by $I_{\nu}$ and $I_{\nu+1}$, respectively, with analogous replacement in case $k=0$ and the Bessel functions evaluated
at $\z_0^{\nu,H_2}$. Further, with the same convention in force, to cover the relevant corner cases we set
$$
b_{n,k}^{\nu,H_1,H_2} :=
	\begin{cases}
		-(-1)^{\chi_{\{k=0\}}}\sqrt{2(\nu+1)} c_k^{\nu,H_2} \z_k^{\nu,H_2} J_{\nu+1}\big(\z_k^{\nu,H_2}\big) & \textrm{if} \;\; H_1+\nu=n=0
			\;\; \textrm{and}\;\; k \ge \chi_{\{H_2+\nu \ge 0\}}, \\
		(-1)^{\chi_{\{n=0\}}}\sqrt{2(\nu+1)} c_n^{\nu,H_1} \z_n^{\nu,H_1} J_{\nu+1}\big(\z_n^{\nu,H_1}\big) & \textrm{if} \;\; H_2+\nu=k=0
			\;\; \textrm{and} \;\; n \ge \chi_{\{H_1+\nu \ge 0\}}.
	\end{cases}
$$

\begin{lemma} \label{lem:main}
Let $\nu > -1$ and $H_1,H_2 \in \mathbb{R} \cup\{\infty\}$ be such that $H_1 \neq H_2$.
Then
$$
\Big\langle \mathbb{L}^{\nu}\psi_n^{\nu,H_1}, \psi_k^{\nu,H_2}\Big\rangle = 
\Big\langle \psi_n^{\nu,H_1}, \mathbb{L}^{\nu}\psi_k^{\nu,H_2}\Big\rangle + b_{n,k}^{\nu,H_1,H_2}, \qquad
	n \ge \chi_{\{H_1+\nu > 0\}}, \quad k \ge \chi_{\{H_2 + \nu > 0\}}.
$$
\end{lemma}

\begin{proof}
We shall use the divergence form of $\mathbb{L}^{\nu}$, see \eqref{id8}, and integrate twice by parts. Denote
$$
\mathcal{I} := \Big\langle \mathbb{L}^{\nu} \psi_n^{\nu,H_1}, \psi_k^{\nu,H_2} \Big\rangle
	= \int_0^1 \mathbb{L}^{\nu}\psi_n^{\nu,H_1}(x) \psi_k^{\nu,H_2}(x)\, dx.
$$
Clearly, the last integral converges since $\mathbb{L}^{\nu}\psi_n^{\nu,H_1}$ (see Proposition \ref{prop:eigen}) and
$\psi_k^{\nu,H_2}$ are in $L^2((0,1),dx)$.

In what follows we assume that $n \ge \chi_{\{H_1+\nu  \ge 0\}}$ and $k \ge \chi_{\{H_2+\nu  \ge 0\}}$.
The remaining corner cases are simpler and left to the reader.
We have
\begin{align*}
\mathcal{I} & = - \int_0^1 x^{-\nu-1/2} \frac{d}{dx} \Big( x^{2\nu+1} \frac{d}{dx} \big[ x^{-\nu-1/2}\psi_n^{\nu,H_1}(x)\big] \Big)
	\psi_k^{\nu,H_2}(x)\, dx \\
& = - x^{\nu+1/2} \frac{d}{dx} \big[ x^{-\nu-1/2} \psi_n^{\nu,H_1}(x)\big] \psi_k^{\nu,H_2}(x) \Big|_0^1 \\
& \qquad + \int_0^1 \frac{d}{dx}\big[ x^{-\nu-1/2}\psi_n^{\nu,H_1}(x)\big] \frac{d}{dx}\big[ x^{-\nu-1/2}\psi_k^{\nu,H_2}(x)\big]
	x^{2\nu+1}\, dx \\
& =  - x^{\nu+1/2} \frac{d}{dx} \big[ x^{-\nu-1/2} \psi_n^{\nu,H_1}(x)\big] \psi_k^{\nu,H_2}(x) \Big|_0^1 \\
& \qquad + x^{\nu+1/2} \psi_n^{\nu,H_1}(x) \frac{d}{dx}\big[ x^{-\nu-1/2} \psi_k^{\nu,H_2}(x)\big] \Big|_0^1 \\
& \qquad - \int_0^1 \psi_n^{\nu,H_1}(x) x^{-\nu-1/2} \frac{d}{dx} \Big( x^{2\nu+1} \frac{d}{dx} \big[ x^{-\nu-1/2}\psi_k^{\nu,H_2}(x)
	\big] \Big)\, dx \\
& \equiv \mathcal{I}_1 + \mathcal{I}_2 + \mathcal{I}_3.
\end{align*}
Since $\mathcal{I}_3 = \big\langle \psi_n^{\nu,H_1}, \mathbb{L}^{\nu}\psi_k^{\nu,H_2} \big\rangle$ it remains to show that
$\mathcal{I}_1 + \mathcal{I}_2 = b_{n,k}^{\nu,H_1,H_2}$.

For symmetry reasons, it suffices to evaluate $\mathcal{I}_1$. Plugging in the explicit expressions for $\psi_n^{\nu,H_1}$
and $\psi_k^{\nu,H_2}$, and then using \eqref{id5}, shows that
$$
\mathcal{I}_1 = c_n^{\nu,H_1} c_k^{\nu,H_2} \z_n^{\nu,H_1} x J_{\nu+1}\big(\z_n^{\nu,H_1}x\big)
	J_{\nu}\big(\z_k^{\nu,H_2}x\big) \Big|_0^1;
$$
here and below we use the same convention concerning the replacement of $J_{\nu}$ and $J_{\nu+1}$ by $I_{\nu}$ and $I_{\nu+1}$
as in the definition of $b_{n,k}^{\nu,H_1,H_2}$. In view of the asymptotics \eqref{as1} we infer that
$$
\mathcal{I}_1 = c_n^{\nu,H_1} c_k^{\nu,H_2} \z_n^{\nu,H_1} J_{\nu+1}\big(\z_n^{\nu,H_1} \big) J_{\nu}\big(\z_k^{\nu,H_2} \big).
$$
This also implies
$$
\mathcal{I}_2 = - c_n^{\nu,H_1} c_k^{\nu,H_2} \z_k^{\nu,H_2} J_{\nu}\big(\z_n^{\nu,H_1}\big) J_{\nu+1}\big(\z_k^{\nu,H_2}\big)
$$
and thus finishes the proof.
\end{proof}

\begin{theorem} \label{thm:main}
Let $\nu > -1$ and assume that $H_1,H_2 \in \mathbb{R} \cup\{\infty\}$ are such that $H_1 \neq H_2$.
Then the following identities hold:
$$
\sum_{k \ge \chi_{\{H_2+\nu > 0\}}}
\frac{\big|b_{n,k}^{\nu,H_1,H_2}\big|^2}{\Big[(-1)^{\chi_{\{n=0\}}} \big(\z_n^{\nu,H_1}\big)^2
	-(-1)^{\chi_{\{k=0\}}}\big(\z_k^{\nu,H_2}\big)^2\Big]^2} = 1,
\qquad n \ge \chi_{\{H_1 + \nu > 0\}}.
$$
\end{theorem}

\begin{proof}
Combining Lemma \ref{lem:main} with Proposition \ref{prop:eigen} we get, for
$n \ge \chi_{\{H_1+\nu > 0\}}$ and $k \ge \chi_{\{H_2+\nu > 0\}}$,
$$
(-1)^{\chi_{\{n=0\}}} \big(\z_n^{\nu,H_1}\big)^2 \Big\langle \psi_n^{\nu,H_1}, \psi_k^{\nu,H_2}\Big\rangle =
(-1)^{\chi_{\{k=0\}}}\big(\z_k^{\nu,H_2}\big)^2 \Big\langle \psi_n^{\nu,H_1}, \psi_k^{\nu,H_2}\Big\rangle + b_{n,k}^{\nu,H_1,H_2}.
$$
Solving for $\big\langle \psi_n^{\nu,H_1}, \psi_k^{\nu,H_2}\big\rangle$ we arrive at
\begin{equation} \label{id111}
\Big\langle \psi_n^{\nu,H_1}, \psi_k^{\nu,H_2}\Big\rangle =
\frac{b_{n,k}^{\nu,H_1,H_2}}{(-1)^{\chi_{\{n=0\}}}\big(\z_n^{\nu,H_1}\big)^2
	- (-1)^{\chi_{\{k=0\}}} \big(\z_k^{\nu,H_2}\big)^2}.
\end{equation}

Here the following remark is in order (recall that $H_1 \neq H_2$): in the above fraction no singularity occurs.
To see this, assume first that $n,k \ge 1$. We may further assume that $H_1,H_2 < \infty$, since if one of them is infinite then
the argument simplifies.
Suppose a contrario that $\z_n^{\nu,H_1} = \z_k^{\nu,H_2} =: \z^{\nu} > 0$. Then, in view of \eqref{id6},
$\z^{\nu}J_{\nu+1}(\z^{\nu}) = (H_1+\nu)J_{\nu}(\z^{\nu}) = (H_2 + \nu)J_{\nu}(\z^{\nu})$. Here the
second equality implies $J_{\nu}(\z^{\nu}) = 0$ and then the first one forces $J_{\nu+1}(\z^{\nu}) = 0$.
But this is a contradiction, since positive zeros of $J_{\nu}$ and $J_{\nu+1}$ are interlaced,
cf.\ \cite[Chap.\,XV, Sec.\,15{$\cdot$}22]{watson}.
If the condition $n,k \ge 1$ does not hold, the only non-trivial case is when $n=k=0$ and $\nu < 0$. But then, supposing that
$\z_0^{\nu,H_1} = \z_0^{\nu,H_2} =: \z^{\nu} > 0$, we would have, see \eqref{id66},
$-\z^{\nu}I_{\nu+1}(\z^{\nu}) = (H_1+\nu)I_{\nu}(\z^{\nu}) = (H_2 + \nu)I_{\nu}(\z^{\nu})$,
which implies $I_{\nu}(\z^{\nu}) = 0$, a contradiction.

Using now Theorem \ref{thm:onb} and Parseval's identity, and then \eqref{id111}, we can write for
any $n \ge \chi_{\{H_1+\nu > 0\}}$
\begin{align*}
1 = \big\| \psi_n^{\nu,H_1} \big\|_{L^2((0,1),dx)}^2 & =
	\sum_{k \ge \chi_{\{H_2+\nu > 0\}}} \Big| \Big\langle \psi_n^{\nu,H_1}, \psi_k^{\nu,H_2}\Big\rangle \Big|^2 \\
& = \sum_{k \ge \chi_{\{H_2+\nu > 0\}}}
\frac{\big|b_{n,k}^{\nu,H_1,H_2}\big|^2}{\Big[(-1)^{\chi_{\{n=0\}}}\big(\z_n^{\nu,H_1}\big)^2
	-(-1)^{\chi_{\{k=0\}}}\big(\z_k^{\nu,H_2}\big)^2\Big]^2}.
\end{align*}

The conclusion follows.
\end{proof}

\section{Special cases} \label{sec:spec}

We now specify the general result of Theorem \ref{thm:main} to a number of cases where the identity in question
takes a more explicit form. As a result, we obtain, up to the best of our knowledge new, formulas \eqref{F1}--\eqref{F18}.
The standing assumption, unless stated otherwise, is $\nu > -1$.
We do not include any computations, all of them being elementary.
The facts needed to perform the computations are \eqref{zeros}, and in some cases \eqref{id20} or \eqref{id30};
the symmetry $b_{n,k}^{\nu,H_1,H_2} = - b_{k,n}^{\nu,H_2,H_1}$ is also useful.

\subsection*{The case $H_1=-\nu$ and $H_{2}=\infty$}
In this case Theorem \ref{thm:main} implies
\begin{equation} \tag{F1} \label{F1}
\sum_{k \ge 1} \frac{\big(\z_k^{\nu}\big)^2}{\Big[\big(\z_n^{\nu+1}\big)^2-\big(\z_k^{\nu}\big)^2\Big]^2} = \frac{1}4, \qquad n \ge 1.
\end{equation}
Moreover, for $n=0$ Theorem \ref{thm:main} recovers Raighley's formula \eqref{Ray}.

\subsection*{The case $H_1=\infty$ and $H_2 = -\nu$}
In this case Theorem \ref{thm:main} gives
\begin{equation} \tag{F2}
\sum_{k \ge 1} \frac{1}{\Big[\big(\z_n^{\nu}\big)^2 - \big(\z_k^{\nu+1}\big)^2\Big]^2} =
	\frac{1}{4\big(\z_n^{\nu}\big)^2} - \frac{\nu+1}{\big(\z_n^{\nu}\big)^4},
	\qquad n \ge 1.
\end{equation}

\subsection*{The case $H_1=\nu$ and $H_2=\infty$}
Here for simplicity we apply Theorem \ref{thm:main} with $\nu > 0$ and then shift the index $\nu-1 \mapsto \nu$ getting,
for $\nu > -1$,
\begin{equation} \tag{F3}
\sum_{k \ge 1} \frac{\big(\z_k^{\nu+1}\big)^2}{\Big[\big(\z_n^{\nu}\big)^2-\big(\z_k^{\nu+1}\big)^2\Big]^2} = \frac{1}4,
	\qquad n \ge 1.
\end{equation}

\subsection*{The case $H_1 = \infty$ and $H_2 = \nu$}
We again use Theorem \ref{thm:main} with the restriction $\nu > 0$ and then shift the index $\nu-1 \mapsto \nu$.
This leads to the formula, valid for $\nu > -1$,
\begin{equation} \tag{F4} \label{F4}
\sum_{k \ge 1} \frac{1}{\Big[ \big(\z_n^{\nu+1}\big)^2-\big(\z_k^{\nu}\big)^2\Big]^2} = \frac{1}{4\big(\z_n^{\nu+1}\big)^2},
	\qquad n \ge 1.
\end{equation}

\subsection*{The case $H_1 = 0$ and $H_2 = \infty$}
Applying Theorem \ref{thm:main} we get
\begin{equation} \tag{F5}
\sum_{k \ge 1} \frac{\big(\z_k^{\nu}\big)^2}{\Big[\big[(\z_n^{\nu})'\big]^2- \big(\z_k^{\nu}\big)^2\Big]^2}
	= \frac{1}4 - \frac{\nu^2}{4\big[(\z_n^{\nu})'\big]^2}, \qquad n \ge 1.
\end{equation}
Moreover, for $\nu <0$ and $n=0$ Theorem \ref{thm:main} gives
\begin{equation} \tag{F6} \label{f6}
\sum_{k \ge 1} \frac{\big(\z_k^{\nu}\big)^2}{\Big[\big[(\z_0^{\nu})'\big]^2 + \big(\z_k^{\nu}\big)^2\Big]^2}
	= \frac{1}4 + \frac{\nu^2}{4\big[(\z_0^{\nu})'\big]^2}, \qquad \nu < 0.
\end{equation}
Note that the case when $\nu=0$ is the same as for the choice $H_1=-\nu$, $H_2=\infty$ considered above.

\subsection*{The case $H_1 = \infty$ and $H_2 = 0$}
Here Theorem \ref{thm:main} produces in case $\nu > 0$
\begin{equation} \tag{F7}
\sum_{k \ge 1}
	\frac{\big[(\z_k^{\nu})'\big]^2}{ \Big[ \big[(\z_k^{\nu})'\big]^2- \nu^2 \Big]
	\Big[ \big(\z_n^{\nu}\big)^2 - \big[(\z_k^{\nu})'\big]^2 \Big]^2}
	= \frac{1}{4\big(\z_n^{\nu}\big)^2}, \qquad n \ge 1, \quad \nu > 0,
\end{equation}
and in case $\nu < 0$
\begin{equation} \tag{F8} \label{f8}
\begin{split}
& \sum_{k \ge 1}
	\frac{\big[(\z_k^{\nu})'\big]^2}{ \Big[ \big[(\z_k^{\nu})'\big]^2- \nu^2 \Big]
	\Big[ \big(\z_n^{\nu}\big)^2 - \big[(\z_k^{\nu})'\big]^2 \Big]^2} \\
& \qquad 	= \frac{1}{4\big(\z_n^{\nu}\big)^2} -
		\frac{\big[(\z_0^{\nu})'\big]^2}{\Big[ \big[(\z_0^{\nu})'\big]^2+\nu^2 \Big]
		\Big[ \big(\z_n^{\nu}\big)^2 + \big[(\z_0^{\nu})'\big]^2\Big]^2},
		\qquad n \ge 1, \quad \nu < 0.
\end{split}
\end{equation}
Observe that the case $\nu=0$ is the same as for the already considered choice $H_1=\infty$, $H_2=-\nu$.

\subsection*{The case $H_1 = -\nu$ and $H_2 = \nu$}
We apply Theorem \ref{thm:main} assuming that $\nu > 0$.
Then in the resulting identity we shift the index $\nu-1 \mapsto \nu$ and get for $\nu > -1$
\begin{equation} \tag{F9}
\sum_{k \ge 1} \frac{1}{\Big[ \big(\z_n^{\nu+2}\big)^2- \big(\z_k^{\nu}\big)^2\Big]^2}
	= \frac{1}{16(\nu+1)^2}, \qquad n \ge 1.
\end{equation}
Moreover, for $n=0$ we recover another formula due to Raighley (cf.\ \cite[Chap.\,XV, Sec.\,15{$\cdot$}51]{watson})
$$
\sum_{k \ge 1} \frac{1}{\big(\z_k^{\nu}\big)^4} = \frac{1}{16(\nu+1)^2 (\nu+2)}.
$$

\subsection*{The case $H_1 = \nu$ and $H_2 = -\nu$}
We again apply Theorem \ref{thm:main} assuming that $\nu > 0$ and shift the index $\nu-1 \mapsto \nu$
in the resulting formula obtaining for $\nu > -1$
\begin{equation} \tag{F10}
\sum_{k \ge 1} \frac{1}{\Big[ \big(\z_n^{\nu}\big)^2 - \big(\z_k^{\nu+2} \big)^2 \Big]^2}
	= \frac{1}{16 (\nu+1)^2} - \frac{\nu+2}{\big(\z_n^{\nu}\big)^4}, \qquad n \ge 1.
\end{equation}

\subsection*{The case $H_1 = -\nu$ and $H_2 = 0$}
In this case from Theorem \ref{thm:main} we get the following formulas. When $\nu > 0$,
\begin{equation} \tag{F11}
\sum_{k \ge 1} \frac{\big[ (\z_k^{\nu})'\big]^2}{\Big[ \big[ (\z_k^{\nu})'\big]^2 - \nu^2 \Big]
	\Big[ \big(\z_n^{\nu+1}\big)^2 - \big[ (\z_k^{\nu})'\big]^2\Big]^2} = \frac{1}{4\nu^2}, \qquad n \ge 1, \quad \nu > 0,
\end{equation}
and for $n=0$ we recover a formula due to Buchholz (cf.\ \cite[Form.\,(74)]{sneddon};
see also the related comments in \cite[p.\,145]{sneddon})
\begin{equation*} 
\sum_{k \ge 1} \frac{1}{\Big[ \big[ (\z_k^{\nu})'\big]^2 - \nu^2 \Big] \big[ (\z_k^{\nu})'\big]^2}
	= \frac{1}{4\nu^2 (\nu+1)}, \qquad \nu > 0.
\end{equation*}
If $\nu < 0$, then
\begin{equation} \tag{F12} \label{f12}
\begin{split}
& \sum_{k \ge 1} \frac{\big[ (\z_k^{\nu})'\big]^2}{\Big[ \big[ (\z_k^{\nu})'\big]^2 - \nu^2 \Big]
	\Big[ \big(\z_n^{\nu+1}\big)^2 - \big[ (\z_k^{\nu})'\big]^2\Big]^2} \\
& \qquad = \frac{1}{4\nu^2}
	-  \frac{\big[ (\z_0^{\nu})'\big]^2}{\Big[ \big[ (\z_0^{\nu})'\big]^2 + \nu^2 \Big]
	\Big[ \big(\z_n^{\nu+1}\big)^2 + \big[ (\z_0^{\nu})'\big]^2\Big]^2}, \qquad n \ge 1, \quad \nu < 0,
\end{split}
\end{equation}
and for $n = 0$
\begin{equation} \tag{F13}
\sum_{k \ge 1} \frac{1}{\Big[\big[(\z_k^{\nu})'\big]^2 - \nu^2 \Big] \big[(\z_k^{\nu})'\big]^2}
	= \frac{1}{4\nu^2 (\nu+1)} - \frac{1}{\Big[\big[(\z_0^{\nu})'\big]^2 + \nu^2 \Big] \big[(\z_0^{\nu})'\big]^2},
		\qquad \nu < 0.
\end{equation}

\subsection*{The case $H_1 = 0$ and $H_2 = - \nu$}
Applying Theorem \ref{thm:main} for $\nu \neq 0$ we get
\begin{equation} \tag{F14}
\sum_{k \ge 1}  \frac{1}{\Big[\big[(\z_n^{\nu})'\big]^2 - \big(\z_k^{\nu+1}\big)^2 \Big]^2}
	= \frac{1}{4\nu^2} - \frac{1}{4\big[(\z_n^{\nu})' \big]^2} - \frac{\nu+1}{\big[(\z_n^{\nu})'\big]^4}, \qquad n \ge 1, \quad \nu \neq 0,
\end{equation}
and the identity corresponding to $n=0$ in case $\nu < 0$ is
\begin{equation} \tag{F15} \label{f15}
\sum_{k \ge 1}  \frac{1}{\Big[\big[(\z_0^{\nu})'\big]^2 + \big(\z_k^{\nu+1}\big)^2 \Big]^2}
= \frac{1}{4\nu^2} + \frac{1}{4\big[(\z_0^{\nu})' \big]^2} - \frac{\nu+1}{\big[(\z_0^{\nu})'\big]^4}, \qquad \nu < 0.
\end{equation}

\subsection*{The case $H_1=\nu$ and $H_2=0$}
Here we apply Theorem \ref{thm:main} with $\nu > 0$ and then shift the index $\nu-1 \mapsto \nu$ in the resulting identity
getting a formula valid for $\nu > -1$
\begin{equation} \tag{F16}
\sum_{k \ge 1} \frac{\big[(\z_k^{\nu+1})'\big]^2}{\Big[ \big[(\z_k^{\nu+1})'\big]^2 - (\nu+1)^2\Big]
	\Big[ \big(\z_n^{\nu}\big)^2 - \big[(\z_k^{\nu+1})' \big]^2\Big]^2} = \frac{1}{4(\nu+1)^2}, \qquad n \ge 1.
\end{equation}

\subsection*{The case $H_1=0$ and $H_2=\nu$}
Applying again Theorem \ref{thm:main} with $\nu > 0$ and shifting the index $\nu-1 \mapsto \nu$ we arrive at an identity
valid for $\nu > -1$
\begin{equation} \tag{F17} \label{F18}
\sum_{k \ge 1} \frac{1}{\Big[\big[(\z_n^{\nu+1})' \big]^2 - \big(\z_k^{\nu}\big)^2 \Big]^2}
	= \frac{1}{4(\nu+1)^2} - \frac{1}{4 \big[(\z_n^{\nu+1})'\big]^2}, \qquad n \ge 1.
\end{equation}

\subsection*{Completely explicit cases}
Finally, we note that formulas \eqref{F1}--\eqref{F4} take explicit forms with the special choice $\nu = -1/2$
(these are the only such possibilities among all formulas \eqref{F1}--\eqref{F18} and all choices of the parameter).
More precisely, taking $\nu = -1/2$, formulas \eqref{F1}--\eqref{F4} can be rewritten, respectively, as
\begin{align*}
\sum_{k \ge 0} \frac{(2k+1)^2}{\big[ (2k+1)^2 - 4n^2\big]^2} & = \frac{\pi^2}{16}, \qquad n \ge 1, \\
\sum_{k \ge 0} \frac{1}{\big[k^2-(n-1/2)^2\big]^2} & = \frac{\pi^2}{4(n-1/2)^2} + \frac{1}{2(n-1/2)^4}, \qquad n \ge 1, \\
\sum_{k \ge 0} \frac{k^2}{\big[k^2 - (n-1/2)^2 \big]^2} & = \frac{\pi^2}4, \qquad n \ge 1, \\
\sum_{k \ge 0} \frac{1}{\big[ (2k+1)^2 - 4n^2 \big]^2} & = \frac{\pi^2}{64 n^2}, \qquad n \ge 1.
\end{align*}
The above four identities are known and can be found in standard compilations, actually as  of more general
formulas, see e.g.\ \cite[Sec.\,5.1.26, Form.\,32]{prud1}, \cite[Sec.\,5.1.25, Form.\,25]{prud1},
\cite[Sec.\,5.1.25, Form.\,35]{prud1}, \cite[Sec.\,5.1.26, Form.\,21]{prud1}, respectively.

\section{Comments and further results} \label{sec:further}

In this section we comment on the size of the ``imaginary'' zero $\z_0^{\nu,H}$, since this is the least studied
in the literature zero we are dealing with. Thus, from now on, we consider $\nu$ and $H$ such that $H+\nu < 0$.

The equation $I_{\nu,H}(x) = 0$ for $x \in (0,\infty)$ can be written equivalently as, see \eqref{id66},
\begin{equation} \label{id:777}
x \frac{I_{\nu+1}(x)}{I_{\nu}(x)} = -(H+\nu).
\end{equation}
Here the left-hand side is continuous and strictly increasing on $(0,\infty)$ (this is seen from its Mittag-Leffler expansion) and
has limits $0$ and $\infty$ as $x \to 0^+$ and $x \to \infty$, respectively. Therefore, given an arbitrary $x_0 > 0$, for a fixed
$\nu$ one can always find $H$ such that $\z_0^{\nu,H} = x_0$. In other words, in general the size of $\z_0^{\nu,H}$ can be arbitrary.

However, the situation is much different in the special case $H=0$ considered in Section \ref{sec:spec}, see
Formulas \eqref{f6}, \eqref{f8}, \eqref{f12} and \eqref{f15}. We will show the following.
\begin{proposition} \label{prop:j0}
Let $-1 < \nu < 0$. Then
$$
(\z_0^{\nu})' < \z_1^{\nu} < \z_1^{\nu+1}.
$$
\end{proposition}

\begin{proof}
The second inequality is well known, see \cite[Chap.\,XV, Sec.\,15{$\cdot$}22]{watson}, so we prove only the first one.

Since $(\z_0^{\nu})'$ is the strictly positive zero of $I_{\nu,0}$, it is, see \eqref{id:777}, the only strictly positive root
of the equation
$$
\frac{I_{\nu+1}(x)}{I_{\nu}(x)} = - \frac{\nu}x.
$$
We now estimate the left-hand side here from below by a function $F(x)$ so that the (first) positive root of
the equation $F(x) = -\nu/x$ is bigger than $(\z_0^{\nu})'$.

Define
$$
F(x):= \frac{2(\nu+2)x}{4(\nu+1)(\nu+2) + x^2}.
$$
From \cite[Theorem 3(a)]{nasell}, see also \cite[p.\,11, formula for $L_{\nu,0,1}$]{nasell}, it follows that
$$
F(x) < \frac{I_{\nu+1}(x)}{I_{\nu}(x)}, \qquad x > 0.
$$
Solving the equation $F(x) = -\nu/x$ we see that its only positive root is
$$
x_0^{\nu} := 2 \sqrt{\frac{(-\nu)(\nu^2 + 3\nu + 2)}{3\nu+4}}.
$$
Thus we get $(\z_0^{\nu})' < x_0^{\nu}$.

On the other hand, from e.g.\ Raighley's formula \eqref{Ray} we know that $\z_1^{\nu} > 2\sqrt{\nu+1}$.
Therefore, to finish the proof it suffices to verify the estimate
$$
x_0^{\nu} <  2 \sqrt{\nu+1}, \qquad \nu \in (-1,0),
$$
which is elementary to check.
\end{proof}

\begin{remark}
The method from the proof of Proposition \ref{prop:j0} together with other rational bounds for $I_{\nu+1}/I_{\nu}$ in
\cite{nasell} allows one to get more precise upper and also lower bounds for $(\z_0^{\nu})'$. Moreover, the same method can be
used to estimate $\z_0^{\nu,H}$ for other $H$ and $\nu$ such that $H+\nu < 0$.
\end{remark}


\end{document}